\documentclass[12pt,leqno,article]{amsart}
\usepackage{amsmath, amssymb, amsthm, mathtools}
\usepackage{xfrac} 
\usepackage{commath}
\usepackage{graphicx, color}
\usepackage{hyperref}
\usepackage{enumerate}
\usepackage{mathrsfs}
\usepackage{enumitem}
\usepackage{bm}
\usepackage[letterpaper,asymmetric,left=1.2in,right=1.2in,top=1.4in,bottom=1.2in,bindingoffset=0.0in]{geometry}

\begin{document}

 \bibliographystyle{plain}
 \newtheorem{theorem}{Theorem}
 \newtheorem{lemma}[theorem]{Lemma}
 \newtheorem{corollary}[theorem]{Corollary}
 \newtheorem{problem}[theorem]{Problem}
 \newtheorem{conjecture}[theorem]{Conjecture}
 \newtheorem{definition}[theorem]{Definition}
 \newtheorem{prop}[theorem]{Proposition}
 \numberwithin{equation}{section}
 \numberwithin{theorem}{section}
\newlist{mylist}{enumerate}{5}
\setlist[mylist]{label=\arabic*}
 \newcommand{\mo}{~\mathrm{mod}~}
 \newcommand{\mc}{\mathcal}
 \newcommand{\rar}{\rightarrow}
 \newcommand{\Rar}{\Rightarrow}
 \newcommand{\lar}{\leftarrow}
 \newcommand{\lrar}{\leftrightarrow}
 \newcommand{\Lrar}{\Leftrightarrow}
 \newcommand{\zpz}{\mathbb{Z}/p\mathbb{Z}}
 \newcommand{\mbb}{\mathbb}
 \newcommand{\B}{\mc{B}}
 \newcommand{\cc}{\mc{C}}
 \newcommand{\D}{\mc{D}}
 \newcommand{\E}{\mc{E}}
 \newcommand{\F}{\mathbb{F}}
 \newcommand{\cF}{\mc{F}}
 \newcommand{\G}{\mc{G}}
 \newcommand{\ZG}{\Z (G)}
 \newcommand{\FN}{\F_n}
 \newcommand{\I}{\mc{I}}
 \newcommand{\J}{\mc{J}}
 \newcommand{\M}{\mc{M}}
 \newcommand{\nn}{\mc{N}}
 \newcommand{\cQ}{\mc{Q}}
 \newcommand{\cP}{\mc{P}}
 \newcommand{\U}{\mc{U}}
 \newcommand{\X}{\mc{X}}
 \newcommand{\Y}{\mc{Y}}
 \newcommand{\itQ}{\mc{Q}}
 \newcommand{\sgn}{\mathrm{sgn}}
 \newcommand{\C}{\mathbb{C}}
 \newcommand{\R}{\mathbb{R}}
 \newcommand{\T}{\mathbb{T}}
 \newcommand{\N}{\mathbb{N}}
 \newcommand{\Q}{\mathbb{Q}}
 \newcommand{\Z}{\mathbb{Z}}
 \newcommand{\A}{\mathbb{A}}
 \newcommand{\ff}{\mathfrak F}
 \newcommand{\fb}{f_{\beta}}
 \newcommand{\fg}{f_{\gamma}}
 \newcommand{\gb}{g_{\beta}}
 \newcommand{\vphi}{\varphi}
 \newcommand{\whXq}{\widehat{X}_q(0)}
 \newcommand{\Xnn}{g_{n,N}}
 \newcommand{\lf}{\left\lfloor}
 \newcommand{\rf}{\right\rfloor}
 \newcommand{\lQx}{L_Q(x)}
 \newcommand{\lQQ}{\frac{\lQx}{Q}}
 \newcommand{\rQx}{R_Q(x)}
 \newcommand{\rQQ}{\frac{\rQx}{Q}}
 \newcommand{\elQ}{\ell_Q(\alpha )}
 \newcommand{\oa}{\overline{a}}
 \newcommand{\oI}{\overline{I}}
 \newcommand{\dx}{\text{\rm d}x}
 \newcommand{\dy}{\text{\rm d}y}
\newcommand{\cal}[1]{\mathcal{#1}}
\newcommand{\cH}{{\cal H}}
\newcommand{\diam}{\operatorname{diam}}
\newcommand{\bx}{\mathbf{x}}
\newcommand{\Ps}{\varphi}
\newcommand{\va}{\bm\alpha}
\newcommand{\vg}{\bm\gamma}
\newcommand{\zp}{\mathbb{Z}_{\textit{p}}}
\newcommand{\qp}{\mathbb{Q}_{\textit{p}}}
\newcommand{\zq}{\mathbb{Z}_{\textit{q}}}
\newcommand{\bA}{\mathbb{A}}
\newcommand{\bN}{\mathbb{N}}
\newcommand{\bQ}{\mathbb{Q}}
\newcommand{\bR}{\mathbb{R}}
\newcommand{\bZ}{\mathbb{Z}}
\newcounter{cases}
\newcounter{subcases}[cases]
\newenvironment{mycases}
  {%
    \setcounter{cases}{0}%
    \setcounter{subcases}{0}%
    \def\case
      {%
        \par\noindent
        \refstepcounter{cases}%
        \textbf{Case \thecases.}
      }%
    \def\subcase
      {%
        \par\noindent
        \refstepcounter{subcases}%
        \textit{Subcase (\thesubcases):}
      }%
  }
  {%
    \par
  }
\renewcommand*\thecases{\arabic{cases}}
\renewcommand*\thesubcases{\roman{subcases}}

\parskip=0.5ex

\title[A three gap theorem for the adeles]{A three gap theorem for the adeles}
\author{Akshat~Das, Alan~Haynes}

\allowdisplaybreaks

\begin{abstract}
We prove a natural generalization of the classical three gap theorem, for rotations on adelic tori. Our proof is an adaptation to the adeles of the lattice based approach to gaps problems in Diophantine approximation originally introduced by Marklof and Str\"ombergsson.
\end{abstract}
\thanks {Research supported by NSF grant DMS 2001248}
\maketitle

\section{Introduction}\label{sec:intro}
The classical three gap theorem (also known as the three distance theorem and as the Steinhaus problem) asserts that, for any $\alpha\in\R$ and $N\in\N$, the collection of points $n\alpha~\mathrm{mod}~1,~1\le n\le N,$ partitions $\R/\Z$ into component arcs having one of at most three distinct lengths. This theorem was first proved independently in the 1950's by S\'{o}s \cite{Sos1,Sos2}, Sur\'{a}nyi \cite{Suranyi58}, and \'{S}wierczkowski \cite{Swier59}, and it has since been reproved numerous times and generalized in many ways (see the introductions and bibliographies of \cite{HaynesMarklof2020a,HaynesMarklof2020b}).

With a view towards understanding problems in dynamics which are sensitive to arithmetic properties of return times to regions, it is desirable to generalize classical results about rotations on $\R/\Z$ to the setting of rotations on adelic tori. In this paper we will prove an adelic version of the three gap theorem. For readers unfamiliar with the adeles or adelic tori, we provide definitions and basic properties in the next section. Here we briefly present our results.

Let $\mc{P}=\{p_1,p_2,\ldots\}$ be a non-empty subset of prime numbers and let $\A_\mc{P}$ denote the projection of the rational adeles $\A$ onto the places indexed by $\{\infty\}\cup\mc{P}$. The additive group $\Gamma_{\mc{P}} = \Z[1/p_1,1/p_2,\ldots]$ can be diagonally embedded into $\A_{\cP}$ as a subgroup, and we identify it with its image under this embedding. The adelic torus $X_\cP$ is then defined as $X_{\cP} = \bA_{\cP} / \Gamma_{\cP}.$ We will write elements $\bm{\alpha}\in X_\cP$ as $\bm{\alpha}=(\alpha_\infty,\alpha_{p_1},\alpha_{p_2},\ldots)$.

We are going to define gaps as nearest neighbor distances, but first we must specify a metric on $X_{\cP}.$ A natural choice of metric on $\A_\cP$ is given by
\begin{equation}\label{eqn.MetricDef1}
	|\bm{\alpha}-\bm{\beta}|=\begin{cases}
		\max\left\{|\alpha_\infty-\beta_\infty|_\infty,\max_{p\in\cP}|\alpha_p-\beta_p|_p\right\}&\text{if }~ |\cP|<\infty,\\ &\\
\max\left\{|\alpha_\infty-\beta_\infty|_\infty,\max_{p\in\cP}\frac{|\alpha_p-\beta_p|_p}{p}\right\}&\text{if }~ |\cP|=\infty.
	\end{cases}
\end{equation}
This metric induces the usual restricted product topology on $\A_\cP$, and we use it to define the metric
\begin{equation}\label{eqn.MetricDef2}
	\|\bm{\alpha}-\bm{\beta}\|=\min\{|\bm{\alpha}-\bm{\beta}-\bm\gamma|:\gamma\in\Gamma_\cP\}
\end{equation}
on $X_\cP$, which induces the quotient topology (see \cite{HaynesMunday2013,TorbaGalindo2013}).

Given $\bm{\alpha} \in \bA_{\cP}$ and $N \in \N$, let
\begin{equation*}
	S_{N}(\bm{\alpha}) = \lbrace \bm\xi_n=n\bm{\alpha} + \Gamma_{\cP}: 1 \leq n \leq N \rbrace \subset X_{\cP},
\end{equation*}
and for each $1\le n\le N$ let $\delta_{n,N}=\delta_{n,N}(\bm\alpha)$ denote the distance from $\bm\xi_n$ to its nearest neighbor in $S_N(\alpha)$. That is,
\begin{equation}\label{eqn.DeltaDef}
\delta_{n,N}=\min\left\{\|\bm\xi_m-\bm\xi_n\|>0 : 1\le m\le N\right\}.
\end{equation}
We are interested in the number of distinct nearest neighbor distances, which we write as
\begin{equation*}
	g_N(\bm\alpha)=|\{\delta_{n,N}(\bm\alpha) \colon 1 \leq n \leq N\}|.
\end{equation*}
The main result of this paper is the following theorem.
\begin{theorem}\label{thm.3GapsAdeles1}
Let $\cP$ be any non-empty set of prime numbers. For any $\bm\alpha\in X_\cP$ and $N \in \N$, we have that $g_N(\bm\alpha) \leq 3$. Furthermore, there exist $\bm\alpha\in X_\cP$ and $N \in \N$ for which $g_N(\bm\alpha) = 3$.
\end{theorem}
Our proof of this theorem is an adaptation to the adeles of the lattice based approach to gaps problems in Diophantine approximation first introduced by Marklof and Str\"ombergsson in \cite{MarklofStrom2017} to give a new proof of the three gap theorem. The utility of their approach lies in its flexibility for generalization to higher dimensional problems where other techniques do not work well (see \cite{HaynesMarklof2020a,HaynesMarklof2020b,HaynesRamirez2020}).

Readers who are familiar with the references in the previous paragraph will recognize the overall structure of our proof of Theorem \ref{thm.3GapsAdeles1}. However, technical details aside, here there are two new difficulties which must be overcome. The first is to prove that a certain function (the function $F$ defined in Section \ref{sec:lattices}) on the space of lattices $\mathrm{SL}(2,\Gamma_{\cP}) \backslash \mathrm{SL}(2,\bA_{\cP})$  is well-defined. For this we use an adelic version of Minkowski's theorem from the geometry of numbers, which was developed independently by McFeat \cite{McFeat1971} and Bombieri and Vaaler \cite{BombieriVaaler1983}. The second difficulty is to prove that the bound of 3 in our theorem is best possible. Of course, for specific choices of $\cP$ this can be done by a computer, but to deal with arbitrary $\cP$ a small amount of ingenuity is required.

The structure of this paper is as follows. In Section \ref{sec.Prelim} we discuss background material and notation, and we establish basic preliminary results. In Section \ref{sec:lattices} we reformulate the problem of bounding $g_N(\bm\alpha)$ as a problem about bounding a certain function on the space of lattices of determinant $\bm 1$ in $\A_\cP^2$. In Section \ref{sec.Proof1/2} we prove the upper bound $g_N\le 3$, and in Section \ref{sec.Proof2/2} we give examples showing that this bound is best possible, in the sense described in Theorem \ref{thm.3GapsAdeles1}.

\section{Preliminaries and notation}\label{sec.Prelim}

For any prime number $p$, we write $\bQ_p$ for the field of $p$-adic numbers and $| \cdot  |_{p}$ for the usual $p$-adic absolute value on this field. The ring of $p$-adic integers $\bZ_p$ is the set of $x \in \bQ_p$ with $\abs{x}_p \leq 1.$ We also use $| \cdot |_{\infty}$ to denote the usual Archimedean absolute value on $\bR$.

The ring of rational adeles $\A$ of $\Q$ is a topological ring consisting of all points of the form
\[\bm\alpha=(\alpha_\infty,\alpha_2,\alpha_3,\alpha_5,\ldots)\in\R\times\prod_p\Q_p,\]
satisfying the condition that $\alpha_p\in\Z_p$ for all but finitely many primes $p$ (the product above is over all prime numbers). Addition and multiplication of elements are defined pointwise, with closure under addition guaranteed by the strong triangle inequality, and the topology on $\A$ is the restricted product topology with respect to the sets $\Z_p\subseteq\Q_p$.

As in the introduction, for a nonempty set of prime numbers $\cP = \left\{p_1, p_2, \ldots \right\}$ we write $\bA_{\cP}$ for the topological ring obtained by projecting $\bA$ onto the coordinates indexed by $\{\infty\}\cup\cP$, and provided with the final topology with respect to this projection. With this topology, the additive group of $\A_\cP$ is a locally compact Abelian group, therefore it has a translation invariant Haar measure which is unique up to scaling. 

The space $\A_\cP$ is metrizable, so there are of course many metrics which induce its topology. The problems that we are studying depend on the choice of metric, and in this paper we choose to use the metric defined by \eqref{eqn.MetricDef1}. For the case when $|\cP|<\infty$ this is the maximum metric, which is a canonical choice. When $|\cP|=\infty$ this is a natural metric, which has been used before in this context \cite{HaynesMunday2013, TorbaGalindo2013}. 

The additive group $\Gamma_{\cP} = \Z[1/p_1,1/p_2,\ldots]$ can be diagonally embedded into $\A_{\cP}$ by the injective homomorphism $\gamma \mapsto \bm\gamma=(\gamma, \gamma, \gamma, \ldots)$, and we identify $\Gamma_\cP$ with its image under this map (in the same way, we also denote the diagonal embedding of any element of $\Q$ into $\A_\cP$ by bold face). The group $\Gamma_\cP$ is a discrete subgroup of $\A_\cP$, and the the quotient group
\begin{equation*}
    X_{\cP} = \bA_{\cP} / \Gamma_{\cP}
\end{equation*}
is compact. The metric $\|\cdot\|$ defined by \eqref{eqn.MetricDef2} induces the quotient topology on $X_\cP$ (this follows from the same arguments given in \cite{HaynesMunday2013}). For clarity of notation, we also mention that there is a natural action of $\Gamma_\cP$ on $\A_\cP$, given by
$\gamma\bm\alpha=\bm\gamma\bm\alpha.$

To help with some of the calculations below, it is worth pointing out that a strict fundamental domain for the quotient group $X_\cP$ can be identified with the set
\[\mc{F}_\cP=[0,1)\times\prod_{p\in\cP}\Z_{p}.\]
The reader should take care to note that this is only a Cartesian product of sets, not a direct product of groups - the group structure is slightly different because of the fact that $\Gamma_\cP$ is \textit{diagonally} embedded in $\A_\cP$.

Finally, we conclude this section with the following useful observation, which is a good exercise in some of the definitions above.
\begin{prop}\label{prop.DoubleBarCheck}
If $\bm\alpha,\bm\beta\in \mc{F}_\cP$ then
\begin{equation}
	\|\bm\alpha-\bm\beta\|=\min_{\gamma\in\{0,\pm 1\}}|\bm\alpha-\bm\beta-\bm\gamma|.
\end{equation}
\end{prop}
\begin{proof}
First of all, since 
\[|\alpha_\infty-\beta_\infty|_\infty< 1\quad\text{and}\quad|\alpha_p-\beta_p|_p\le 1~\text{for all}~p\in\cP,\]
we have that
\begin{equation}\label{eqn.FundDomProp1}
	\|\bm\alpha-\bm\beta\|\le|\bm\alpha-\bm\beta|\begin{cases}
		\le 1&\text{if}~|\cP|<\infty,\\
		< 1&\text{if}~|\cP|=\infty.
	\end{cases}
\end{equation}
Choose $\gamma\in\Gamma_\cP$ so that
\[\|\bm\alpha-\bm\beta\|=|\bm\alpha-\bm\beta-\bm\gamma|.\]
If, for some prime $p\in\cP$, we had $|\gamma|_p>1$ then it would follow that
\[|\alpha_p-\beta_p-\gamma|_p\ge p,\]
This would give that
\[\|\bm\alpha-\bm\beta-\bm\gamma\|\begin{cases}
	> 1&\text{if}~|\cP|<\infty,\\
	\ge 1&\text{if}~|\cP|=\infty,
\end{cases}\]
which would contradict \eqref{eqn.FundDomProp1}. It follows that $\gamma\in\Z_p$ for all $p\in\cP$, which implies that $\gamma\in\Z$. Now if $|\gamma|_\infty\ge 2$, then we would have that
\[\|\bm{\alpha}-\bm\beta\|\ge|\alpha_\infty-\beta_\infty-\gamma|_\infty> 1,\]
again contradicting \eqref{eqn.FundDomProp1}. Therefore $\gamma=-1, 0,$  or 1, as required.
\end{proof}

\section{Formulation of the problem in terms of lattices}\label{sec:lattices}
Let $G=\mathrm{SL}(2,\bA_{\cP})$ and $\Gamma=\mathrm{SL}(2,\Gamma_{\cP})$. The goal of this section is to explain how the quantity $g_N(\bm\alpha)$ can be obtained as the value of a function on the quotient space $\Gamma \backslash G$. This space can be identified in a natural way with the space of lattices of determinant $\bm{1}$ in $\A_\cP^2$, since such a lattice is determined as the $\Gamma_\cP$-span of the rows of a matrix in $G$, which is unique up to left multiplication by an element of $\Gamma$ (i.e. change of basis).

Suppose that $\bm\alpha\in\A_\cP$ and $N\in\N$, and write $N_+=N+1/2$. Beginning from definition \eqref{eqn.DeltaDef}, for each $1\le n\le N$ we have
\begin{align*}
	\delta_{n,N}(\bm\alpha)&=\min\left\{|(m-n)\bm\alpha+\bm\gamma|>0:0< m< N_+,\gamma\in\Gamma_\cP\right\}\\
	&=\min\left\{|k\bm\alpha+\bm\gamma|>0:-n< k< N_+-n,\gamma\in\Gamma_\cP\right\}\\
	&=\min\left\{|k\bm\alpha+\bm\gamma|>0:\frac{-n}{N_+}< k< 1-\frac{n}{N_+},\gamma\in\Gamma_\cP\right\}.
\end{align*}
For each non-zero $t\in\Q$ define
\begin{equation}\label{matrixA}
    A_t(\va) = 
    \begin{pmatrix}
    \bm{1} & \va \\
    0 & \bm{1}
  \end{pmatrix}
  \begin{pmatrix}
    \bm{t}^{-1} & 0 \\
    0 & \bm{t}
  \end{pmatrix}
  =
   \begin{pmatrix}
    \bm{t}^{-1} & t\va \\
    0 & \bm{t}
  \end{pmatrix}
  \in G,
\end{equation}
and note that, for any $\beta,\gamma\in\Gamma_\cP$,
\[(\bm\beta, \bm\gamma) A_{N_+}(\va) = \left(\frac{\bm\beta}{N_+}, N_+(\beta\va + \bm\gamma) \right).\]
It follows from this and the computation above that, for each $1\le n\le N$, the quantity $\delta_{n,N}$ is equal to
 \begin{align}
&\frac{1}{N_{+}}\min \Bigg\{ |\bm{v}| \neq 0 \colon (\bm{u}, \bm{v}) \in \Gamma_{\cP}^2 A_{N_{+}}(\va), ~\frac{-n}{N_{+}} < u_{\infty} < 1 - \frac{n}{N_{+}}, \label{eqn.DeltForm1}\\ 
& \hspace*{100bp}  |u_p|_p \leq \frac{1}{|N_+|_p} ~\text{for all}~ p \in {\cP} \Bigg\}. \nonumber
 \end{align}
Next, for any $M\in G, ~t\in (0,1)$, and $z\in \N$, let us define
\begin{align}\label{eqn.QDef}
	Q(M,t,z)=&\left\{(\bm{u},\bm{v})\in\Gamma_\cP^2M:\bm{v}\not=0,
	~-t<u_\infty <1-t,\phantom{\frac{2}{z}}\right.\\
	&\left.\hspace*{100bp}|u_p|_p\le\left|\frac{2}{z}\right|_p~\text{for all}~ p \in {\cP}\right\},\nonumber
\end{align}
and
\begin{align}\label{eqn.FDef}
	F(M,t,z)=\min\left\{|\bm{v}|:(\bm{u},\bm{v})\in Q(M,t,z)\right\}.
\end{align}
The reason for these definitions will be made clear below. However, before proceeding further, we must verify the following proposition.
\begin{prop}\label{prop.FWellDef}
The quantity $F$ is well-defined as a function from $\Gamma \backslash G \times (0,1)\times\N$ to $\R_{>0}$.
\end{prop}
\begin{proof}
First of all let $(M,t,z)\in G \times (0,1)\times\N$ and choose $\epsilon >0$ small enough so that:
\begin{itemize}
	\item[(i)] $\epsilon<\min\{t,1-t\},$\vspace*{5bp}
	\item[(ii)] $p\epsilon<|2/z|_p$, for all primes $p|2z$, and\vspace*{5bp}
	\item[(iii)] $\Gamma_\cP^2M$ contains no non-zero lattice points $(\bm{u},\bm{v})$ satisfying $|\bm{u}|\le\epsilon$ and $|\bm{v}|=0$.
\end{itemize}
Condition (iii) is possible because of the uniform discreteness of the lattice.

Now write
\[\mc{S}=\left\{(\bm\alpha, \bm\beta)\in\A_{\cP}^2:|\bm\alpha|<\epsilon\right\},\]
and suppose that $(\bm\alpha, \bm\beta)\in\mc{S}$. Then from condition (i) we have that
\[-t<\alpha_\infty<1-t.\]
If $|\cP|<\infty$ then for all primes $p\in\cP$ we have from (ii), together with the fact that $\epsilon<1$, that
\[|\alpha_p|_p\le \epsilon < \left|\frac{2}{z}\right|_p.\]
If $|\cP|=\infty$ then for primes $p\in\cP$ with $p|2z$ we have from (ii) that
\[|\alpha_p|_p\le p\epsilon < \left|\frac{2}{z}\right|_p.\]
In this case for primes $p\in\cP$ with $p\nmid 2z$ we use the discreteness of the $p$-adic absolute value to deduce that
\[\frac{|\alpha_p|_p}{p}\le \epsilon <1 ~\Rightarrow~ |\alpha_p|_p\le 1=\left|\frac{2}{z}\right|_p.\]
These arguments show that
\[\mc{S}\cap \Gamma_\cP^2M\subseteq Q(M,t,z).\]
Since $\mc{S}$ is a symmetric and convex subset of $\A_\cP^2$ with infinite Haar measure, it follows from the adelic analogue of Minkowski's convex body theorem (see \cite[Section III]{BombieriVaaler1983}) that $\mc{S}$ contains a non-zero element of $\Gamma_\cP^2 M$. Therefore $Q(M,t,z)$ is non-empty. 

The existence of the minimum in the definition of $F$ follows from the uniform discreteness of the lattice $\Gamma_\cP^2 M$, and this also guarantees that $F$ never takes the value 0. Finally, since the same lattice is determined by choosing any other representative for $M$ from $\Gamma\backslash G$, the function $F$ is well-defined on $\Gamma \backslash G \times (0,1)\times\N.$
\end{proof}
Comparing equations \eqref{eqn.DeltForm1}-\eqref{eqn.FDef}, we have that
\begin{equation*}
	\delta_{n,N}(\va) = \frac{1}{N_{+}} F\left(A_{N_{+}}(\va), \frac{n}{N_{+}},2N+1 \right).
\end{equation*}
Motivated by this observation, for $M\in G$ and $z\in\N$ we define
\begin{equation*}%\label{eqn.GDef}
\cal{G}(M,z) = |\{F(M,t,z) : 0 < t < 1\}|,
\end{equation*}
and for $N\in\N$, we also set
\begin{equation*}%\label{eqn.GNDef}
	\cal{G}_N(M) = \left|\left\{F\left(M,\frac{n}{N_{+}},2N+1\right) : 1 \leq n \leq N  \right\}\right|.
\end{equation*}
It follows that
\begin{equation}\label{eqn.gNBnd1}
    g_N(\va) = \cal{G}_N(A_{N_{+}}) \leq \cal{G}(A_{N_{+}},2N+1).
\end{equation}
This reduces the problem of finding an upper bound for the number of gaps to that of finding an upper bound for the function $\mc{G}$. We conclude this section with the following basic observation.
\begin{prop}\label{prop.GFinite}
	For any $M\in G$ and $z\in\N$, we have that $\mc{G}(M,z)<\infty$.
\end{prop}
\begin{proof}
Proposition \ref{prop.FWellDef} implies that there exists a vector $(\bm{u},\bm{v})\in Q(M,1/2,z)$. By symmetry, we also have that $(-\bm{u},-\bm{v})\in Q(M,1/2,z)$. For any $t\in(0,1)$, one of $\pm u_\infty$ lies in the interval $(-t,1-t)$, and so one of the vectors $\pm(\bm{u},\bm{v})$ lies in $Q(M,t,z)$. It follows that $F(M,t,z)\le |\bm{v}|$, for all $t\in (0,1)$. By the uniform discreteness of the lattice $\Gamma_\cP^2 M$, there are finitely many vectors $(\bm{u}',\bm{v}')$ in the set
\begin{equation}\label{eqn.QMzDef}
	Q(M,z)=\bigcup_{t\in(0,1)}Q(M,t,z),
\end{equation}
satisfying the condition $|\bm{v}'|\le |\bm{v}|$. Therefore the set of values taken by the function $F(M,t,z)$, as $t$ varies over $(0,1)$, is a finite set.
\end{proof}

\section{Proof of Theorem \ref{thm.3GapsAdeles1}, part 1}\label{sec.Proof1/2}
In this section we will prove the following result.
\begin{theorem}\label{thm.3GapsLats}
Let $\cP$ be a non-empty set of prime numbers. For any $M\in G$ and $z\in\N$, we have that $\mc{G}(M,z)\le 3$.
\end{theorem}
In view of inequality \eqref{eqn.gNBnd1}, this theorem implies the upper bound in the statement of Theorem \ref{thm.3GapsAdeles1}.

To establish Theorem \ref{thm.3GapsLats}, suppose that $M\in G$ and $z\in\N$, and let $Q(M,z)$ be defined as in \eqref{eqn.QMzDef}. Note that, by the symmetry of the lattice and of the definition of $Q(M,z)$,
\begin{equation}\label{eqn.QSym}
	(\bm{u},\bm{v})\in Q(M,z)\quad\Leftrightarrow\quad(-\bm{u},-\bm{v})\in Q(M,z).
\end{equation}

By Proposition \ref{prop.GFinite}, there exists a number $K\in\N$ such that $\mc{G}(M,z)=K$. It is clear from definitions that we can fix vectors $(\bm{u}_1,\bm{v}_1), \dots, (\bm{u}_{K},\bm{v}_{K}) \in Q(M,z)$ for which the following properties hold:
\begin{itemize}
	\item[(V1)] $0<\abs{\bm{v}_1}<\abs{\bm{v}_2}<\dots<\abs{\bm{v}_K}$,\vspace*{5bp}
	\item[(V2)] For each $t \in (0,1),$ there exists $1 \leq i \leq K$ such that  $F(M,t,z)=\abs{\bm{v}_i},$ and \vspace*{5bp}
	\item[(V3)] For each $1 \leq i \leq K$, there exists $t \in (0,1)$ such that $(\bm{u}_i,\bm{v}_i) \in Q(M,t,z)$ and $F(M,t,z)=\abs{\bm{v}_i}$.
\end{itemize}
We also have the following proposition.
\begin{prop}
If $\mc{G}(M,z)=K$ then we can choose the vectors $(\bm{u}_i,\bm{v}_i)$ as above, so that they satisfy conditions (V1)-(V3), and so that $u_{i,\infty}\ge 0$ for each $1\le i\le K$.
\end{prop}
\begin{proof}
Suppose that $1\le i\le K$ and that $u_{i,\infty}<0$. It is clear that the vector $(-\bm{u}_i,-\bm{v}_i)$ satisfies property (V2), and we wish to show that it also satisfies (V3). Since $(\bm{u}_i,\bm{v}_i)$ itself satisfies (V3), there exists a number $t\in (0,1)$ with $(\bm{u}_i,\bm{v}_i) \in Q(M,t,z)$ and $F(M,t,z)=\abs{\bm{v}_i}$. This implies that there are no vectors $(\bm{u},\bm{v})\in Q(M,z)$ with $|\bm{v}|<|\bm{v}_i|$ and $u_\infty\in (-t,1-t)$. Writing $t'=1-t$ and using \eqref{eqn.QSym}, we see that there are also no vectors $(\bm{u},\bm{v})\in Q(M,z)$ with $|\bm{v}|<|\bm{v}_i|$ and $u_\infty\in (-t',1-t')$. Since $(-\bm{u}_i,-\bm{v}_i) \in Q(M,t',z)$, this implies that $F(M,t',z)=|-\bm{v}_i|$, and we see that (V3) holds for this vector.

It follows that, for each $1\le i\le K$ with $u_{i,\infty}<0$, we can replace the vector $(\bm{u}_i,\bm{v}_i)$ by its negative, to obtain a new list of vectors satisfying the conclusion of the proposition.
\end{proof}
We will henceforth assume, without loss of generality, that the vectors $(\bm{u}_i,\bm{v}_i)$ have been chosen as above, so that they also satisfy:
\begin{itemize}
	\item[(V4)] $u_{i,\infty}\ge 0$ for each $1\le i\le K$.
\end{itemize} 
Next we have the following proposition.
\begin{prop}\label{prop.Smallu_i}
If $(\bm{u},\bm{v}) \in Q(M,z)$ and $|u_\infty|<1/2$, then $F(M,t,z)\le |\bm{v}|$ for all $t\in (0,1)$.
\end{prop}
\begin{proof}
By replacing $(\bm{u},\bm{v})$ with its negative if necessary, we may assume without loss of generality that $u_\infty\in [0,1/2)$. We then have that $(\bm{u},\bm{v}) \in Q(M,t,z)$ for all $t\in (0,1-u_{\infty})$, and that $(-\bm{u},-\bm{v}) \in Q(M,t,z)$ for all $t\in(u_{\infty},1)$. Therefore $F(M,t,z)\le |\bm{v}|$ for all $t$ in the union of these two intervals. If $u_{\infty}<1/2$, then the union of these two intervals is all of $(0,1)$, and the statement of the proposition follows.
\end{proof}
Finally, we have the following proposition.
\begin{prop}\label{prop.u_iOrder}
If $1\le i\le K$ and if $(\bm{u},\bm{v})\in Q(M,z)$ satisfies $|u_\infty|_\infty\le u_{i,\infty}$, then $|\bm{v}|\ge|\bm{v}_i|$.
\end{prop}
\begin{proof}
Suppose that the hypotheses are satisfied and, by replacing $(\bm{u},\bm{v})$ by its negative if necessary, suppose that $u_\infty\ge 0$. By property (V3), there exists a $t\in (0,1)$ such that $(\bm{u}_i,\bm{v}_i) \in Q(M,t,z)$ and $F(M,t,z)=\abs{\bm{v}_i}$. Then since
\[-t<0\le u_\infty\le u_{i,\infty}<1-t,\]
we also have that $(\bm{u},\bm{v})\in Q(M,t,z)$. This implies that $F(M,t,z)\le |\bm{v}|,$ which gives the desired conclusion.
\end{proof}
Note that Proposition \ref{prop.u_iOrder} implies that
\[0\le u_{K,\infty}<u_{K-1,\infty}<\cdots <u_{1,\infty}<1.\]
Now we are ready to complete the proof of Theorem \ref{thm.3GapsLats}. Let $K_1$ denote the number of indices $1\le i\le K$ with $u_{i,\infty}<1/2$, and let $K_2$ denote the number of indices with $1/2\le u_{i,\infty}<1$.

By Proposition \ref{prop.Smallu_i}, together with properties (V1) and (V2), we have that $K_1\le 1$, and if $K_1=1$ then the corresponding index $i$ is equal to $K$.

If $K_2\le 2$ then clearly we have that $K=K_1+K_2\le 3$. Therefore suppose that $K_2\ge 3$, and let $1\le i,j,k\le K$ be the smallest three indices with $u_{i,\infty},u_{j,\infty},u_{k,\infty}\ge 1/2.$ Without loss of generality, by relabeling if necessary, we may assume that $i<j$ and that $v_{i,\infty}$ and $v_{j,\infty}$ are either both negative, or both non-negative. This guarantees that
\[|v_{i,\infty}-v_{j,\infty}|_\infty\le\max\{|v_{i,\infty}|_\infty,|v_{j,\infty}|_\infty\},\]
and by using the strong triangle inequality at the non-Archimedean places we obtain the bound
\begin{equation}\label{eqn.vivjDiff}
	|\bm{v}_i-\bm{v}_j|\le \max\left\{|\bm{v}_i|,|\bm{v}_j|\right\}=|\bm{v}_j|.
\end{equation}
However, it is also the case that
\[|u_{i,\infty}-u_{j,\infty}|_\infty<1/2\le u_{j,\infty},\]
and that 
\[|u_{i,p}-u_{j,p}|_p\le\max\left\{|u_{i,p}|_p,|u_{j,p}|_p\right\}\le\left|\frac{2}{z}\right|_p,\] for all $p\in\cP$. It follows that the vector $(\bm{u},\bm{v})=(\bm{u}_i,\bm{v}_i)-(\bm{u}_j,\bm{v}_j)$ lies in $Q(M,z)$ and satisfies $|u_\infty|<1/2$. By Proposition \ref{prop.Smallu_i}, we must have that
\[F(M,t,z)\le |\bm{v}|=|\bm{v}_i-\bm{v}_j|,\] for all $t\in (0,1)$. Combining this with \eqref{eqn.vivjDiff}, and with (V1), we conclude in this case that $j=K$, $K_1=0, K_2=3$, and $K=K_1+K_2=3$. This completes the proof of Theorem \ref{thm.3GapsLats}, and also the proof of the upper bound in the statement of Theorem \ref{thm.3GapsAdeles1}.

\section{Proof of Theorem \ref{thm.3GapsAdeles1}, part 2}\label{sec.Proof2/2}
To complete the proof of Theorem \ref{thm.3GapsAdeles1} we must show that, for any choice of $\cP$, there are examples of $\bm\alpha\in X_\cP$ and $N\in\N$ for which $g_N(\bm\alpha)=3$. Since the definition of the metric in \eqref{eqn.MetricDef1} depends on whether $|\cP|$ is finite or infinite, we will consider these two cases separately. In the examples below, we make repeated implicit use of Proposition \ref{prop.DoubleBarCheck}.\vspace*{5bp}

\noindent {\bf Finite case ($\bm{|\cP|<\infty}$).} Suppose that $|\cP|<\infty$ and consider the following examples:
\begin{itemize}
	\item[(F1)] If $\cP = \lbrace 2 \rbrace $, take $\va = (\alpha_{\infty},\alpha_2) = (351/100,1)$ and $N = 52$. By direct computation we have that
	\begin{align*}
		\delta_{1,N}(\va)&= \|51\va\| = \frac{1}{100},\\
		\delta_{2,N}(\va)&= \|35\va\| = \frac{3}{20},\quad \text{and}\\
		\delta_{18,N}(\va)&= \|16\va\| = \frac{4}{25},    
	\end{align*} 
	which gives $g_N(\va) = 3$.\vspace*{5bp}
	\item[(F2)] If $\cP = \lbrace 3 \rbrace $, take $\va = (\alpha_{\infty},\alpha_3) = (16/5,1)$ and $N = 5$. Again by direct computation, we have that
	\begin{align*}
		\delta_{1,N}(\va)&= \|4\va\| = \frac{1}{5},\\
		\delta_{2,N}(\va)&= \|3\va\| = \frac{3}{5},\quad \text{and}\\
		\delta_{3,N}(\va)&= \|\va\| = \frac{4}{5},    
	\end{align*} 
	which gives $g_N(\va) = 3$.\vspace*{5bp}
	\item[(F3)] If $\cP=\{p_1,\ldots ,p_k\}$, with $p_1<\cdots <p_k$ and $p_1\cdots p_k\ge 5$, let $N=p_1\cdots p_k+1$ and
	\[\va=\left(\frac{1}{4p_1\cdots p_k},-1,\ldots ,-1\right).\]
	For each $1\le n\le N-3$ we have that
	\[-n\not= 0~\mathrm{mod}~p_i,~\text{for some}~1\le i\le k,\]
	and that
	\[-n\not= 1~\mathrm{mod}~p_j,~\text{for some}~1\le j\le k.\]
	It follows from Proposition \ref{prop.DoubleBarCheck} that $\|n\bm\alpha\|=1$ for all such $n$. From this we see that
	\begin{align*}
		\delta_{1,N}(\va)&= \|(p_1\cdots p_k)\va\| = \max\left\{\frac{1}{4},\frac{1}{p_1}\right\}\le 1/2,\\
		\delta_{2,N}(\va)&= \|(p_1\cdots p_k-1)\va\| = \frac{3}{4}+\frac{1}{4p_1\cdots p_k},\quad \text{and}\\
		\delta_{3,N}(\va)&= \|\bm\alpha\|=1,    
	\end{align*} 
	which gives $g_N(\va) = 3$. It is worth mentioning that one reason this construction does not work in the cases described in the previous two examples is because the corresponding value of $N$ is too small.
\end{itemize}
It is clear that examples (F1)-(F3) cover all possibilities with $|\cP|<\infty.$\vspace*{5bp}

\noindent {\bf Infinite case ($\bm{|\cP|=\infty}$).} Suppose that $|\cP|=\infty$ and consider the following examples:
\begin{itemize}
	\item[(I1)] If the smallest prime in $\cP$ is $3$ then let $\alpha_{\infty}=1/9, \alpha_3=1,$ and $\alpha_p=0$ for all $p\in\cP$ with $p\not=3$, and take $N = 11$. Then we have that
	\begin{align*}
		\delta_{1,N}(\va)&= \|10\va\| = \begin{cases}
			\frac{1}{5}&\text{if}~5\in\cP,\\
			\frac{1}{9}&\text{if}~5\notin\cP,
		\end{cases}\\
		\delta_{2,N}(\va)&= \|7\va\| = \frac{2}{9},\quad \text{and}\\
		\delta_{5,N}(\va)&= \|\va\| = \frac{1}{3},    
	\end{align*} 
	which gives $g_N(\va) = 3$.\vspace*{5bp}
	\item[(I2)] If $\cP$ contains the primes $2$ and $3$ then let $\alpha_{\infty}=27/50, \alpha_2=-1,$ and $\alpha_p=0$ for all $p\in\cP$ with $p\not=2$, and take $N = 6$. Then we have that
	\begin{align*}
		\delta_{1,N}(\va)&= \|5\va\| = \frac{3}{10},\\
		\delta_{2,N}(\va)&= \|4\va\| = \frac{1}{3},\quad \text{and}\\
		\delta_{3,N}(\va)&= \|\va\| = \frac{23}{50},    
	\end{align*} 
	which gives $g_N(\va) = 3$. We note that in the calculation of $\delta_{1,N}$ and $\delta_{2,N}$, it is important that $\alpha_3=0$. \vspace*{5bp}
	\item[(I3)] If $\cP$ contains the prime $2$ but not the prime $3$ then let $\alpha_{\infty}=8/49, \alpha_2=-1,$ and $\alpha_p=0$ for all $p\in\cP$ with $p\not=2$ or 5. Also let $\alpha_5=3$ if $5\in\cP$, and take $N = 8$. Then we have that
	\begin{align*}
		\delta_{1,N}(\va)&= \|7\va\| = \frac{1}{7},\\
		\delta_{2,N}(\va)&= \|5\va\| = \frac{1}{4},\quad \text{and}\\
		\delta_{4,N}(\va)&= \|2\va\| = \frac{16}{49},    
	\end{align*} 
	which gives $g_N(\va) = 3$. The assumption on $\alpha_5$ (if $5\in\cP$) is important in the calculation of $\delta_{1,N}$.\vspace*{5bp}
	\item[(I4)] If the smallest prime $q$ in $\cP$ is greater than or equal to $5$ then let $\alpha_{\infty}=\frac{q-1}{q(q-2)}, \alpha_q=-1,$ and $\alpha_p=0$ for all $p\in\cP$ with $p\not=q$, and take $N = q$. Then, by the type of argument given in example (F3) above, we have that
	\begin{align*}
	\delta_{1,N}(\va)&= \|(q-1)\va\| \\&= \max\left\{\frac{1}{q(q-2)},\max\left\{\frac{1}{p}:p\in\cP,p\not=q\right\}\right\}<\frac{1}{q},\\
	\delta_{2,N}(\va)&= \|(q-2)\va\| = \frac{1}{q},\quad \text{and}\\
	\delta_{3,N}(\va)&= \|\va\| = \frac{1}{q}+\frac{1}{q(q-2)},    
	\end{align*} 
which gives $g_N(\va) = 3$. 
\end{itemize}
Examples (I1)-(I4) cover all possibilities with $|\cP|=\infty.$ This therefore completes the proof of Theorem \ref{thm.3GapsAdeles1}.

\bibliography{GapsAdeles1}

\vspace{.15in}

{\footnotesize
	\noindent
	Department of Mathematics\\University of Houston,\\
	Houston, TX, United States.\\
	atdas@math.uh.edu, haynes@math.uh.edu\\
	}

\end{document}